\documentclass[12pt]{article}
\usepackage{amsmath}
\usepackage{amssymb}
\usepackage{amsthm}
\usepackage{epsfig} 
\usepackage{wrapfig}
\usepackage{pdfpages} 
\usepackage{xcolor}
\usepackage[small]{caption}
\usepackage{subcaption}
\usepackage[pdftex,colorlinks]{hyperref}
\newcommand {\R}{\mathbb R}
\newcommand{\p}{\mathbb P}
\newcommand{\e}{\mathbb E}
\newcommand{\ce}{{\cal E}}

\newcommand {\eq}[1]{\begin {equation} #1 \end {equation}}
\newcommand {\eqn}[1]{\begin {equation*} #1 \end {equation*}}

\renewcommand {\d}{\text{d}}

\newcommand {\f}{F}

\newcommand{\g}{f}

\newcommand\lk{\left[}
\newcommand\rk{\right]}
\renewcommand\lg{\left\{}  
\newcommand\rg{\right\}}
\newcommand\lp{\left(}
\newcommand\rp{\right)}

\renewcommand {\sf}{{\cal F}}

\theoremstyle{definition}\newtheorem{thm}{Theorem}[section]
\theoremstyle{definition}\newtheorem{lem}[thm]{Lemma}
\theoremstyle{definition}\newtheorem{cor}[thm]{Corollary}
\theoremstyle{definition}
\theoremstyle{definition}\newtheorem{defi}[thm]{Definition}
\theoremstyle{definition}\newtheorem{rem}[thm]{Remark}
\theoremstyle{definition}\newtheorem{prop}[thm]{Proposition}
\theoremstyle{definition}
\theoremstyle{definition}
\numberwithin{equation}{section}
\renewcommand {\sp}{\R^d}
\DeclareMathOperator* {\argmin}{argmin}

\newcommand{\al}{\alpha}
\newcommand{\no}[1]{|{#1}|}
\newcommand{\norm}[1]{||{#1}||}
\newcommand{\m}{\mathbf M}
\newcommand{\n}{\mathbf N}
\newcommand{\N}{{\cal N}}

\newcommand{\ufai}{U_n}
\newcommand{\ufak}{U_k}
\newcommand{\ufa}{U}

\newcommand{\ufuk}{U_k}
\newcommand{\ufu}{U}

\newcommand{\sfai}{R_n}
\newcommand{\sfa}{R}
\newcommand{\cfui}{C^{\al}_n}

\newcommand{\cfu}{C^{\al}}

\newcommand{\Pn}{\Phi_n}
\newcommand{\Prn}{\Phi_n^r}

\newcommand{\xfai}{X_n}

\newcommand{\xfui}{X_n}

\newcommand{\ol}{\overline}
\newcommand{\cp}{{\cal P}}

\newcommand{\fp}{{\p_\Phi}}
\newcommand{\fe}{{\e_\Phi}}

\newcommand{\fpi}{\p^{\g}_\Phi}

\newcommand{\cfpn}{\cp^{\g,n}_0}
\newcommand{\cfp}{{\cp_0}}
\newcommand{\cfpz}{{\cp^{\g,0}_0}}

\newcommand{\cfpi}{\cp^{\g}_0}

\newcommand{\ccfpn}{\widetilde\cp^{\g,n}_0}

\newcommand{\ccfpi}{\widetilde\cp^{\g}_0}

\newcommand{\cfe}{\ce_0}

\newcommand{\Z}{\mathbb Z}
\newcommand{\q}{\mathbb Q}
\newcommand{\cfq}{{\cal Q}_0}
\newcommand{\cq}{{\cal Q}}
\newcommand{\id}{\text{id}}

\newcommand{\ind}{\mathbf 1}

\newcommand{\fst}{S}

\newcommand{\fua}{d_{\al}}

\newcommand{\om}{\omega}

\newcommand{\ir}{\Bbb{R}}

\newcommand{\bbN}{\Bbb{N}}

\newcommand{\Om}{\Omega}

\newcommand{\pom}{point-shift}
\newcommand{\Pom}{Point-shift}
\newcommand{\PoM}{Point-Shift}
\newcommand{\pomk}{{point-map}}
\newcommand{\PoMk}{{Point-Map}}
\newcommand{\Pomk}{{Point-map}}
\newcommand{\php}{(\Phi,\p)}

\newcommand{\thetg}{\theta_\g}

\newcommand{\Thetg}{(\theta_{\g})_*}

\newcommand{\tov}{\stackrel{v}{\to}}
\newcommand{\tow}{\stackrel{w}{\to}}
\newcommand{\tod}{\stackrel{d}{\to}}
\newcommand{\bob}{{\cal B}_b}
\newcommand{\bo}{{\cal B}}
\newcommand{\ps}{\Phi_*}

\begin{document}

\author{Fran\c cois Baccelli\thanks{baccelli@math.utexas.edu} \\{\small University of Texas at Austin} 
\and Mir-Omid Haji-Mirsadeghi\thanks{mirsadeghi@sharif.ir}\\{\small Sharif University of Technology}}

\title{Point-Map-Probabilities of a Point Process \\
and Mecke's Invariant Measure Equation}
\date{}
\maketitle
\begin{abstract}
A compatible \pom{} $\f$ maps, in a translation invariant way,
each point of a stationary point process $\Phi$ to some point of $\Phi$.
It is fully determined by its associated \pomk{}, $\g$, which gives 
the image of the origin by $\f$. It was proved by
J. Mecke that if $\f$ is bijective, then the Palm probability of
$\Phi$ is left invariant by the translation of  $-\g$. 
The initial question motivating this paper is the following
generalization of this invariance result: in 
the non-bijective case, what probability measures
on the set of counting measures are left invariant 
by the translation of $-\g$?  The \pomk{}-probabilities of 
$\Phi$ are defined from the action of the semigroup of \pomk{} 
translations on the space of Palm probabilities, and more precisely from 
the compactification of the orbits of this semigroup action.
If the \pomk{}-probability exists, is uniquely defined, and if
it satisfies certain continuity properties,
it then provides a solution to this invariant measure problem. 
\Pomk{}-probabilities are objects of independent interest.
They are shown to be a strict generalization of Palm probabilities:
when $\f$ is bijective, the \pomk{}-probability
of $\Phi$ boils down to the Palm probability of $\Phi$.
When it is not bijective, there exist cases where the \pomk{}-probability
of $\Phi$ is singular with respect to its Palm probability.
A tightness based criterion for the existence of the \pomk{}-probabilities
of a stationary point process is given.
An interpretation of the \pomk{}-probability as the conditional law of the
point process given that the origin has $\f$-pre-images of all orders is
also provided. The results are illustrated by a few examples.
\end{abstract}


\section*{Introduction}

A \emph{\pom{}} is a mapping which is defined on all discrete
subsets $\phi$ of $\sp$ and maps each point $x\in\phi$ to some
point $y\in\phi$; i.e., if $\f$ is a \pom, for all discrete
$\phi\subset\sp$ and all $x\in\phi$, $\f(\phi,x)\in\phi$.
Bijective point-shifts were studied in a seminal paper by
J. Mecke in \cite{Me75}. 
The concept of point-map was introduced by
H. Thorisson (see \cite{Th00} and the references therein).
Points-maps were further studied 
by M. Heveling and G. Last 
\cite{HeLa05}. 
The latter reference also contains 
a short proof of Mecke’s invariance theorem.
\Pom{}s are also known as allocation rules (see e.g. \cite{LaTh09}).
A \pom{} is \emph{compatible with the translations of $\sp$}
or simply \emph{compatible} if
\eqn{\forall t\in\sp, \quad \f(\phi+t,x+t)=\f(\phi,x)+t.}
As will be seen, a translation invariant \pom{} $\f$ is
fully determined by its \pomk{} $\g$ which associates to 
all $\phi$ containing the origin the image of the latter by $\f$, i.e.,
$\g(\phi)=\f({\phi},0)$.
The \pom{} $\f$ is called bijective on the point process $\Phi$ if,
for almost all realizations $\phi$ of the point process,
$\f(\phi,\cdot)$ is bijective on the set $\phi$.

The Palm probability of a translation invariant point process
$\Phi$ is often intuitively described as the distribution of $\Phi$ conditionally
on the presence of a point at the origin. 
This definition was formalized by C. Ryll-Nardzewski \cite{RN} based on
the Matthes definition of Palm probabilities 
(see e.g. \cite{baccelli03elements}). This is the so called
{\em local interpretation} of the latter.
The presence of a point at the origin makes
the Palm distribution of $\Phi$ singular with respect to (w.r.t.) the
translation invariant distribution of $\Phi$.

The present paper is focused on the
\pomk{}-probabilities (or the $\g$-probabilities) of $\Phi$.  
Under conditions described in the paper, the $\g$-pro\-ba\-bi\-li\-ties can
be described as the law of $\Phi$ conditionally
on the event that the origin has $\f$-pre-images of all 
orders (Theorem \ref{thmloc}). This event 
is not of positive probability in general,
and hence it is not possible to define this conditional probability
in the usual way.

The first aim of this paper is to make this definition rigorous.
The proposed construction is based on dynamical system theory.
The action of the semigroup of translations by $-f$
on probability distributions on counting measures
having a point at the origin is considered;
the $f$-probabilities of $\Phi$ are then defined as the $\omega$-limits
of the orbit of this semigroup action on the Palm distribution of $\Phi$
(Definition \ref{deffond}).
As the space of probability
distributions on counting measures is not compact, the existence
of $f$-probabilities of $\Phi$ is not granted.
A necessary and sufficient conditions for their existence 
is given in Lemma \ref{kalen}. Uniqueness is not granted either.
An instance of construction of the $f$-probabilities of Poisson point processes
where one has existence and
uniqueness is given in Theorem \ref{regenver}.

It is shown in Section \ref{summary} that,
when they exist, \pomk{}-probabilities
generalize Palm probabilities.
A key notion to see this is that of
{\em evaporation}. One says that there is evaporation
when the image of $\Phi$ by the
$n$-th iterate of $F$ tends to the empty counting measure
for $n$ tending to infinity.

When there is no evaporation,
the $f$-probabilities of $\Phi$ are just the Palm
distributions of $\Phi$ w.r.t. certain translation invariant thinnings of $\Phi$
and they are then absolutely continuous w.r.t. the Palm distribution
$\cfp$ of $\Phi$; in particular, if $\f$ is bijective, then
the $f$-probability of $\Phi$ exists, is uniquely defined, and coincides with ${\cal P}_0$.
However, in the evaporation case,
the $f$-probabilities of $\Phi$ do not admit a representation 
of this type and they are actually singular w.r.t. $\cfp$
(Theorem \ref{theevasin}).

It is also shown in Theorem \ref{continv}
that, under appropriate continuity properties on $f$, 
a certain mixture of the $f$-probabilities
of $\Phi$ is left invariant by the shift of $-f$.
This generalizes Mecke's point stationarity theorem which states that
if $\f(\Phi,\cdot)$ is bijective and 
if $\Phi$ is distributed according to ${\cal P}_0$,
then so is $\Phi-f$.

Section \ref{pernot} contains
the basic definitions and notation used in the paper, together
with a small set of key examples.
Section \ref{summary} gathers the main results and proofs.
Several more examples are discussed in Section~\ref{proofs}.
The basic tools of point process theory and
dynamical system theory used in the paper are 
summarized in the appendix.

\section{Preliminaries and Notation}\label{pernot}
 
\subsection{General Notation} Each measurable mapping 
$h:(X,{\cal X})\to(X^\prime,{\cal X}^\prime)$
between two measurable spaces induces a measurable mapping
$h_*:M(X)\to M(X^\prime)$, where $M(X)$ is the set of
all  measures on $X$: 
if $\mu$ is a measure on $(X,{\cal X})$, $h_*\mu$ is the
measure on $(X',{\cal X}')$ defined by 
\eq{\label{indmeasu}h_*\mu(A):= (h_*\mu)(A)=\mu(h^{-1} A) .}
Note that if $\mu$ is a probability measure, $h_*\mu$ is also
a probability measure. 

\subsection{Point Processes}
Let $\n=\n(\sp)$ be the space of all locally finite counting measures
(not necessarily simple) on $\sp$.
One can identify
each element of $\n$ with the associated multi-subset of $\sp$.
The notation $\phi$ will be used
to denote either the measure or the related multi-set.
Let $\N$ be the
Borel $\sigma$-field with respect
to the vague topology on the space of counting measures 
(see Subsection \ref{RandMeas} in appendix for more on this subject). 
The measurable space $(\n,\N)$ is the
\emph{canonical space} of point processes.  

The \emph{support} of a counting measure $\phi$ is the same set
as the multi-set related to $\phi$, but without the multiplicities, and
it is denoted by $\ol\phi$. 
The set of all counting measure supports
is denoted by $\ol\n$, i.e.,
$\ol\n$ is the set of all \emph{simple} counting measures.
$\N$ naturally induces a $\sigma$-field $\ol\N$ on $\ol\n$.
 
Let $\n^0$ (respectively, $\overline\n^0$) denote the set of
all elements of $\n$ (respectively, $\overline\n$) which contain
the origin, i.e., for all $\phi\in \n^0$
(respectively, $\phi\in\overline\n^0$), one has $0\in\phi$.

A $\emph{point process}$ is  a
couple $\php$ where $\p$ is a probability measure on a
measurable space $(\Om,\sf)$ and $\Phi$ is a measurable mapping
from $(\Om,\sf)$ to $(\n,\N)$.
If $(\Om,\sf)=(\n,\N)$ and $\Phi$ is the identity on $\n$, 
the point process is defined on the canonical space.
Calligraphic letters $\cp,\cq,\ldots$
(resp. blackboard bold letters $\p,\q,\ldots$) will be used
for probability measures defined on the canonical space 
(resp. on $(\Om,\sf)$). 
The \emph{canonical version} of a point process $(\Phi,\p)$ is
the point process $(\id,\ps\p)$ which is defined on the canonical space.
Here $\id$ denotes the identity on $\n$.

\subsection{Stationary Point Processes}
Whenever $(\sp,+)$ acts (in a measurable way) on a space, 
the action of $t\in\sp$ on that space will be denoted by $\theta_t$. 
It is assumed that $(\sp,+)$ acts on the reference
probability space $(\Omega,\sf)$, or equivalently
that this space is equipped with a measurable
flow $\theta_t:\Omega\rightarrow \Omega$, with $t$ ranging over $\sp$. 
This is a family of mappings such that
$(\om,t)\mapsto\theta_t\om$ is measurable, $\theta_0$ is the identity
on $\Om$ and \eqn{\theta_s\circ\theta_t = \theta_{s+t}.}
 
A point process $\Phi$ is then said to be \emph{compatible} if
\eq{\label{comp.crit}\Phi(\theta_t\omega,B-t)=\Phi(\omega, B), \quad \forall\omega\in\Omega, t\in\sp,B\in{\cal B},}
where by convention, 
$\Phi(\om,B):=(\Phi(\om))(B).$ 
Here ${\cal B}$ denotes the Borel $\sigma$-algebra on $\R^d$.

The action $\theta_t$ of $t\in \sp$ can also be
used on the space of counting measures to denote the translation by $-t$.
For a counting measure $\phi\in\n(\sp)$,
$\theta_t\phi$ is then the counting measure defined by 
$ \theta_t\phi(B)=\phi(B+t).$
Using this notation, the compatibility criterion (\ref{comp.crit})
can be rewritten as 
\eqn{\Phi\circ\theta_t=\theta_t\circ \Phi.}
Note that for consistency reasons,
the action $\theta_t$ of $t\in \sp$ on $\sp$ itself is then
$\theta_t x= x-t,\quad \forall x\in \sp.$

The probability measure $\p$ on $(\Omega,\sf)$ is $\theta_t$-invariant if 
$(\theta_{t})_*\p=\p$.
If, for all $t\in\sp$, $\p$ is $\theta_t$-invariant, it is called stationary.
Below,  a \emph{stationary point process}
is a point process $(\Phi,\p)$ such that $\Phi$ is compatible
and $\p$ is stationary. 

When the point process is simple and stationary with a non-degenerate 
(positive and finite) intensity, its
\emph{Palm probability} is a classical object in the literature.

The Palm probability
of a general (i.e., not necessarily simple) point process $\Phi$
is defined by
\eq{\label {pag}\fp[A] :=\frac 1 {\lambda \no B} \int_{\Omega}
\int_{B}  \mathbf {1} \{\theta_t\om\in A\} \Phi(\om,\d t)\p [\d\om],}
for all $A \in {\sf}$,
and for all Borel sets $B\subset \R^d$ with a non-degenerate
(positive and finite)
Lebesgue measure. Note that the multiplicity of the atoms
of $\Phi$ is taken into account in the last definition.
If a point process $(\Phi,\p)$ is stationary and
has a non-degenerate intensity, the pair
$(\Phi,\fp)$ is called the \emph{Palm version}
of $(\Phi,\p)$.
Expectation w.r.t. $\fp$ will be denoted  by $\fe$.

Whenever the context specifies
a reference point process $(\Phi,\p)$, the short notation
$\cp$ will be used to denote its distribution: i.e., $\cp=\ps\p$.
If in addition, $\Phi$ is stationary and with a non-degenerate intensity,
the distribution of its Palm version
will be denoted by $\cfp$, i.e., $\cfp=\ps\fp$, 
and expectation  w.r.t. $\cfp$ will be denoted by ${\cal E}_0$. 
In the canonical setup, the Palm version
of $(\Phi,\p)=(\id,\cp)$ is
$(\Phi,\fp)=(\id,\cfp)$.

\subsection{Compatible \PoM s} 
\label{sec:cpoms}
\subsubsection{Point-Maps}
A \pom{} on $\n$ is a measurable function
$\f:\n\times \R^d \to \R^d$,
which is defined for all pairs $(\phi, x)$, where $\phi\in\n$
and $x\in\phi$, and satisfies the relation
$\f(\phi,x) \in \phi$ for all $x\in \phi$.

In order to define compatible \pom{}s,
it is  convenient to use the notion of \emph{\pomk}.
A measurable function $\g:\n^0 \to \sp$
is called a \pomk{} if for all 
$\phi$ in $\n^0$,
one has $\g(\phi)=\g(\ol\phi)$,
i.e., it depends only on $\ol\phi$, and
if $\g(\phi)\in \ol\phi$.
 
If $\g$ is a  \pomk, the associated
\emph{compatible \pom, $\f$, is}
$$\f(\phi,x) = \g(\theta_x\phi) + x
  =\theta_{-x} \g(\theta_x\phi) .$$
The \pom{} $\f$ is compatible in the sense that
\begin{eqnarray}\label{compf}
\f(\theta_t\phi,\theta_tx) &=& 
\f(\theta_t\phi, x -t)
=\g(\theta_{x -t}(\theta_t\phi))+x-t\nonumber\\
&=&\g(\theta_x\phi)+x-t=\f(\phi,x)-t=
\theta_t(\f(\phi,x)).
\end{eqnarray}
In the rest of this article, \pom{} always means  compatible \pom{}. 
Small letters will be used for \pomk s and capital letters for
the associated \pom s.

For the \pomk{} $\g$, the \emph{action of the \pomk} on $\n^0(\sp)$
will be denoted by $\thetg$ and defined by 
\eqn{\forall \phi\in\n^0(\sp); \quad\thetg(\phi)=\theta_{\g(\phi)}(\phi).}

\subsubsection{Iterates of a Point-Shift}
\label{sec:cpoms-ite}
For all $n\ge 0$, all $\phi\in  \n$ and $x\in \phi$,
the $n$-th order iterate of the point-shift $\f$
is defined inductively
by $\f^0(\phi,x)=x$ and
$$\f^{k+1}(\phi,x)=
\f (\phi, \f^k(\phi,x)), \quad k\ge 0.
$$
For all $n$, $F^n$ is a compatible \pom{} and
the associated point-map,
which will be denoted by $\g^n$, satisfies
\eq{\label{nthpom} \g^n(\phi)= 
\g^{n-1}(\phi) + \g(\theta_{f^{n-1}}(\phi)),\quad n\ge 1,}
with $\g^0(\phi)=0$ and $\phi\in\n^0$. It is easy to verify that for all $n\in \bbN$, on $\n^0$,  
\eqn{\theta_{\g^n}= \thetg^n,}
and hence 
\eq{\label{semigroupg}\theta_{\g^{m+n}}=\theta_{\g^m}\circ\theta_{\g^n}.}

In accordance with
the definition of $F^n$, for all $n\ge 1$, let
\eqn{F^{-n}(\phi,x)=\{y\in\phi; F^n(\phi,y)=x\}.}
\subsubsection{Image Point Processes}
\label{sec:cpoms-mul}

Let $f$ be a \pomk. For all $\phi \in \n$ and all nonnegative
integers $n$, let
\eq{\label{phifn}
m_\g^n(\phi,y)= \phi(F^{-n}(\phi,y))=
\sum_{x\in F^{-n}(\phi,y)}\phi(\{x\}),
\quad \forall y\in \phi,}
where by convention, the summation over the empty set is zero.
Note that if $\phi$ is simple, then
$ m_\g^n(\phi, y)=\mathrm{card}(F^{-n}(\phi,y)).$
\begin{defi}
Assume that $m_\g^n(\phi,y)<\infty$ for all $y\in \phi$.
The $n$-th image counting measure (of $\phi$ by $\f$) is then
defined as the counting measure $\phi^n_\g$ with support 
$\{y\in\phi;F^{-n}(\phi,y)\neq \emptyset\}$,
and such that the multiplicity of $y$ in the support of $\phi_f^n$
is $m_\g^n(\phi,y)$.
\end{defi}

It will be shown below that,
for all stationary point processes $(\Phi,\p)$,
for all $n\ge 0$, $(\Phi^n_f,\p)$ is a stationary point process
(item 1 in Remark \ref{remfincond})
with the same intensity as $\Phi$
(item 2 in Remark \ref{remfincond}).
The point process $\Phi^n_f$ will be referred to as the $n$-th {\em image
point process} (of $\Phi$ by the \pomk{}).

\subsection{First \PoM{} Examples}
This subsection presents a few basic examples of \pom{}s.
These examples will allow one to illustrate the main
results in Section \ref{summary}.
More details on these examples and further examples can
be found in Section \ref{proofs}.

\subsubsection{Strip  \PoM{}} \label{sr0}
The strip \pom{} was introduced by Ferrari,
Landim and Thorisson \cite{FeLaTh04}. 
For all points $x=(x_1,x_2)$ in the plane, let $T(x)$ denote the 
half strip $(x_1,\infty)\times [x_2-1,x_2+1]$. 
Then $S(\phi,x)$ is the
left most point of $\phi$ in $T(x)$ (see Figure \ref{fig0}).
It is easy to verify that $S$ is
compatible. It is not bijective.
Its \pomk{} will be denoted by $s$.

\begin{rem}
The strip \pom{} is not well defined when there are more than one
left most point in $T(x)$, or when there is no
point of $\phi$ in $T(x)$.
However there is no problem if we consider the strip  \pom{}
(and all other \pom s) on point processes for which the \pom{}
is almost surely well-defined. Note that these
two difficulties can always be taken care of by fixing,
in some translation invariant manner, 
the choice of the image in the case of ambiguity,
and by defining $F(\phi, x)=x$
in the case of non-existence. By doing so one gets a
\pom{} defined for all $(\phi,x)$.
\end{rem}

\subsubsection{Strip \PoM{} on the Random Geometric Graph} \label{drgg}
The \emph{strip \pomk{} on the random geometric graph}
with the neighborhood radius $r$ is
\eqn{g(\phi)=
\begin{cases}
s(\phi) & \norm{s(\phi)}<r\\
0 & \text{otherwise,}
\end{cases}
}
where $s$ is the strip \pomk{}.
The associated \pom{} is depicted in Figure
\ref{fig0}. It will be denoted by $G$. It is not bijective.
Its \pomk{} will be denoted by $g$.

\begin{figure}[h]
\centering
{{\includegraphics[width=0.47\linewidth]{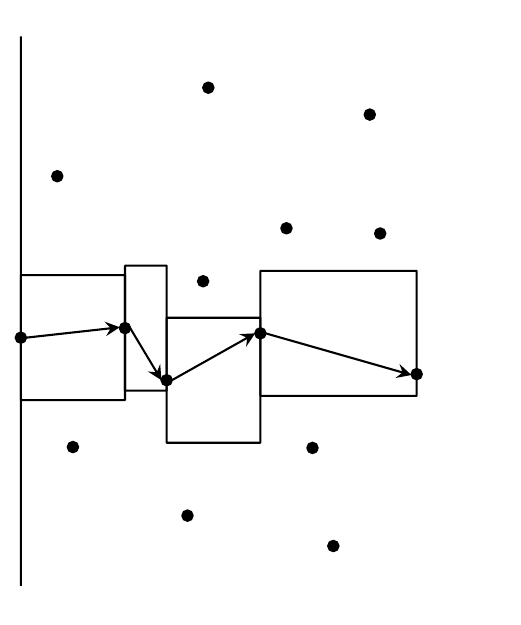}}
{\includegraphics[width=0.47\linewidth]{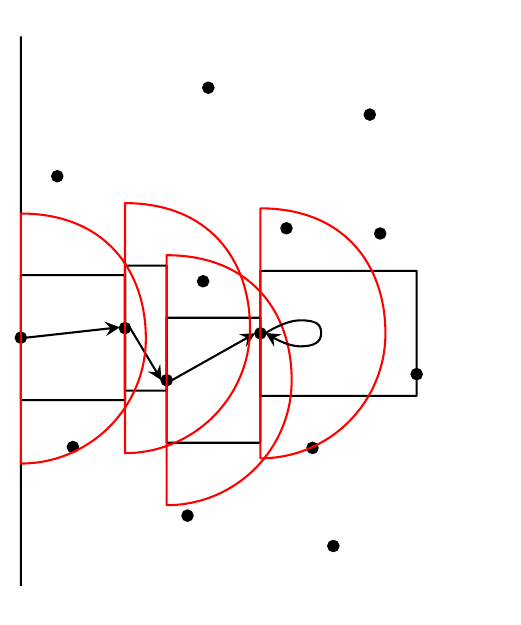}}
\caption{Left: Iterates of the strip \pom{} $S$.
Right: Iterates of $G$, the strip \pom{} on the random geometric graph.
In both cases, the point
$G^4(\phi,0)$ is that at the end of the directed path.}
\label{fig0}}
\end{figure}

\subsubsection{Closest \PoM}
The \emph{closest \pom{}}, $C$, maps each point of
$x\in \phi$ to the point $y\ne x$ of $\phi$ which is the closest. 
This \pom{} is not bijective either.
The associated \pomk{} will be denoted by $c$.
It is depicted in Figure \ref{fig01}.

\subsubsection{Mutual-Neighbor \PoM}
The \emph{mutual-neighbor \pom{}, $N$,} 
maps each point
$x\in\phi$ to the point $y$ of $\phi$ which is the
closest to $x$, if $x$ is also the point of $\phi$ which is the
closest to $y$. Otherwise, it maps $x$ to itself.
It is easy to see that $N$ is bijective and involutive:
$N^2\equiv\id$. The associated \pomk{} will be denoted by $n$.
It is depicted in Figure \ref{fig01}.

\begin{figure}[h]
\centering
{\fbox{\includegraphics[width=0.47\linewidth]{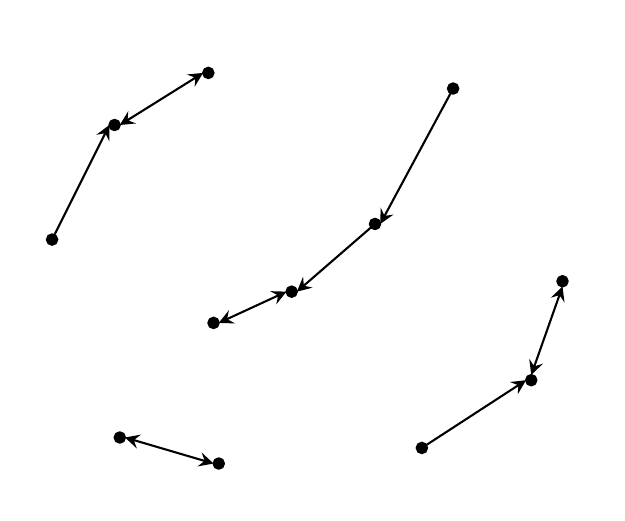}}
\fbox{\includegraphics[width=0.47\linewidth]{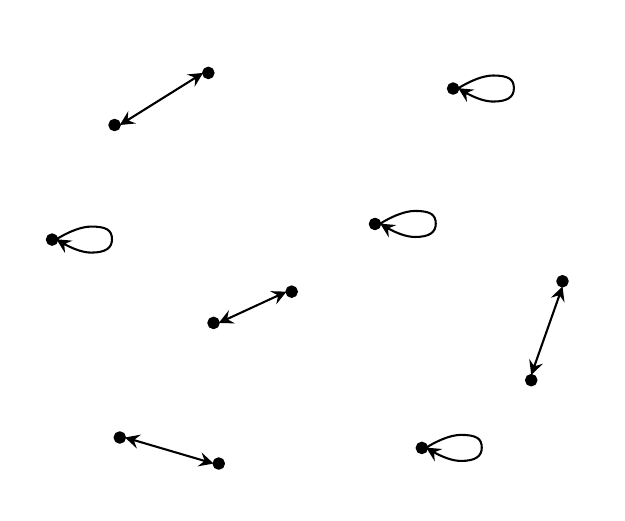}}
\caption{Left: the closest \pom{} $C$.
Right: the mutual-neighbor \pom{} $N$.
The directed edge emanating from a point indicates the image of the point.
\label{fig01}}
}
\end{figure}

\subsection{Mecke's Point Stationarity Theorem}
One of the motivations of this work is to extend  the following
proposition proved by J. Mecke in \cite{Me75}.
\begin{thm}[Point Stationarity]\label{point.sta}
Let $(\Phi,\p)$ be a simple stationary point process and 
let $\f$
be a \pom{} such that $\f(\Phi,\cdot)$ is $\p$-a.s. bijective.
Then the Palm probability of the point-process
is invariant under the action of $\theta_{f}$; i.e.,
\eq{\label{second.thet}\fp=(\theta_{\g(\Phi)})_*\fp,}
with
$\theta_{\g(\Phi)}$ seen as a map from
$\Omega$ to itself defined by
\eqn{\label{first.thet}
(\theta_{\g(\Phi)})(\om):=\theta_{\g(\Phi(\om))}\om. 
}
Since $\fp[\Phi(\{0\})>0]=1$, $\theta_{\g(\Phi)}$
is $\fp$-almost surely well defined.
\end{thm}
\begin{rem}
The fact that $\thetg$ is bijective $\Phi_*\fp$-a.s.
is equivalent to the fact that $F$ is 
bijective on $\ps\p$-almost all realizations
of the point process.
\end{rem}

\section{Results}\label{summary}

\subsection{Semigroup Actions of a \PoMk}
Below, 
$\n^0=\n^0(\R^d)$
and $\m^1(\n^0)$ denotes the set of probability measures
on $\n^0$. 
For all \pomk s $f$ on $\n^0$,
consider the following actions $\pi=\{\pi_n\}$ of $(\bbN,+)$:

\begin{enumerate}
\item $X=\n^0$,  equipped with the vague topology, and
for all $\phi\in \n^0$ and $n\in \bbN$,
$$\pi_n(\phi)=\thetg^n(\phi)\in \n^0,$$
where $\thetg^n$ is defined in Subsection~\ref{sec:cpoms-ite}.
\item $X=\m^1(\n^0)$, equipped with the weak convergence of probability measures on $\n^0$,
and for all ${\cal Q}\in \m^1(\n^0)$ and $n\in \bbN$,
$$\pi_n({\cal Q})=\Thetg^n{\cal Q}=(\thetg^n)_*{\cal Q}\in \m^1(\n^0).$$ 
\end{enumerate}
\subsection{Periodicity and Evaporation}

The \pomk{} $f$
will be said to be {\em periodic} on the stationary point process
$(\Phi,\p)$ if for $\ps\fp$-almost all $\phi$,
the action of $\thetg^n$ is periodic on $\phi$, namely if
there exists integers $p=p(\phi)$ and  $K=K(\phi)$
such that for all $n\ge K$, $\thetg^n(\phi)=\thetg^{n+p}(\phi)$. 
The case where $p$ is independent of $\phi$ is known as 
$p$-periodicity. 
The special case of $1$-periodicity is that where,
$\thetg^n(\phi)$ is 
{\em stationary} (in the dynamical system sense) after some steps,
i.e., such that  for all $n>K(\phi)$, $\thetg^n(\phi) = \thetg^K(\phi)$.
Note that if for all 
$x\in\phi$, the trajectory
$F^n(\phi,x)$ is stationary,
i.e., such that  for all $n>K(\phi,x)$, $\f^n(\phi,x)=\f^K(\phi,x)$,
then $\g$ is 1-periodic.

The mutual-neighbor \pomk{} $n$ on a homogeneous
Poisson point process is 2-periodic.

Similarly, for the closest \pomk{} $c$, the iterates of this
point-shift form a {\em descending chain}, namely a sequence
of point of the support of the point process such that the
distance between the $k+1$-st and the $k$-th
is non-increasing in $k\ge 0$.
The well known fact that there are no infinite
descending chains in the homogeneous Poisson point process
(see \cite{DaLa05}) implies that $c$ is 2-periodic on such a point process,
with the points of the period being mutual-neighbors.

If $g$ is the strip \pomk{} on the random geometric graph
defined in Subsection~\ref{drgg}, 
the strong Markov property of the stationary Poisson point process
on $\sp$ (see \cite{Zu06} for details on the strong Markov property of Poisson point process)   gives that the point process on the right half-plane of
$G(0)$ is distributed as the original Poisson point process.
Hence $G$ is a.s. 1-periodic, even
when the underlying random geometric graph is supercritical.

\begin{rem}
Note that there are other ways of defining periodicity, possibly leading to
other periods. For instance, for the mutual-neighbor \pomk{}
on a Poisson point process, the sequence of image
point processes $\{\Phi^f_n\}_{n\ge 0}$ 
(defined in Subsection \ref{sec:cpoms-mul})
is 1-periodic whereas $f$ is 2-periodic according to
the definition proposed above.
\end{rem}

The point process $(\Phi,\p)$ will be said to {\em evaporate}
under the action of the \pomk{} $f$ if $\ol{\Phi^n_f}$
converges a.s. to the null measure as $n$ tends to infinity,
i.e., for $\p$-almost surely, the set
\begin{equation}
\label{eq:definf}
\ol{\Phi_f^\infty}:=\bigcap_{n=1}^\infty \ol{\Phi_f^n}
\end{equation}
is equal to the empty set
(note that $\ol{\Phi^n_f}$ is a non increasing sequence of sets).
Consider the following set 
\begin{eqnarray}
\label{eq:III}
I &:= & \{\phi\in\n^0; \forall n\in\Bbb N,  \f^{-n}(\phi,0)\neq\emptyset\}
\nonumber \\
& = & \{\phi\in\n^0; \forall n\in\Bbb N,  m^n_f(\phi,0)>0\}
\end{eqnarray}
(see Subsection \ref{sec:cpoms-ite} for
the definition of $\f^{-n}(\phi,y)$ and Subsection 
\ref{sec:cpoms-mul} for that of $m^n_f$).
\begin{lem}\label{eveq}
For all \pomk s $f$, the 
stationary point process $(\Phi,\p)$
evaporates under the action of $f$ if and only if $\fp[\Phi\in I]=0$. 
\end{lem}
\begin{proof}
Let $\cp=\ps\p$ and  $\cfp=\ps\fp$. If $\chi(\phi,x)$
is the indicator of the fact that $x$ has $\f$-pre-images of all orders,
then $\chi$ is a compatible marking of the point process
(i.e., $\chi(\phi,x)= \chi(\theta_x \phi,0)$
for all $x\in \phi$).
Therefore if $\Psi$ denotes the sub-point process of the points 
with mark $\chi$ equal to 1, then $(\Psi,\p)$ is a stationary
point process and by Campbell's theorem,
\eq{\label{simmark}\lambda_\Psi=\lambda_\Phi\e_\Phi[\chi(\Phi,0)]
=\lambda_\Phi\fp[\Phi\in I].}
The evaporation of $(\Phi,\p)$ by $f$ means that
$\Psi$ has zero intensity.
According to  (\ref{simmark}) this is equivalent to $\fp[\Phi\in I]=0$.
\end{proof}

The homogeneous Poisson point process
on ${\mathbb R}^2$ evaporates under the
action of the strip \pomk{} $s$ (see Section \ref{proofs}).

\subsection{Action of $\Thetg$}
\subsubsection{Image Palm Probabilities}
Let $\Phi$ be a stationary point process on $\sp$ and
$f$ be a \pomk.
Consider the action of $\Thetg$ (see Equation (\ref{indmeasu}))
when $\cq=\cp_0$, the Palm distribution of $\Phi$.
It follows from the definition
and from (\ref{pag}) that,
for all $n\ge 1$, for all $G\in\N$
and for all Borel sets $B$ with
non-degenerate Lebesgue measure
\begin{eqnarray}
\label{eq:vbas}
(\thetg^n)_*\cfp[G]
& = & \frac 1 {\lambda \no B}
\int_{\n} \int_B
\mathbf {1} \{\thetg^n\circ \theta_t (\phi) \in G\}
\phi(\d t)
\cp [\d\phi].
\end{eqnarray}
In what follows,
$\cfpn$ is a short notation for
the probability on $\n^0$ defined in the last equation.
This probability will be referred to as the 
\emph{$n$-th image Palm probability (w.r.t. $\g$)} of the point process.

It follows from the semigroup property (\ref{semigroupg}) that
\eq{\label{cpntocpn+1}\Thetg\cfpn=\cp_0^{\g,n+1},
\quad \forall n\in \bbN ,}
when letting $\cfpz:=\cfp$.
From the mass transport relation \cite{LaTh09}, and
using the image counting measure $\phi_f^n$ defined
in Subsection~\ref{sec:cpoms-mul}, one gets:
\begin{lem}\label{convfrom0}
For all $n\ge 0$, and all $G\in {\cal N}$,
\begin{eqnarray}
\cfpn [G]
& = & \frac 1 {\lambda \no B}
\int_{\n}
\int_{B}
\mathbf {1} \{\theta_t\phi\in G\} \phi^n_f (\d t)
\cp [\d\phi].
\label {cffpan}
\end{eqnarray}
\end{lem}
Note that, in general, the $n$-th image Palm probability
$\cfpn$ is {\em not} the Palm probability of
the $n$-th image point process $\Phi^n_f$ (which is the distribution
of $\Phi^n_f$ given that the origin belongs to $\Phi^n_f$ when using
the local interpretation of the Palm probability).
It is rather is the distribution of $\Phi$ given
that the origin is in the $n$-th image process. In both cases,
point multiplicities should be taken into account.
\begin{rem}\label{remfincond}
Equation (\ref{cffpan}) has several important implications:
\begin{enumerate}
\item If $\cal P$ is the distribution of a simple stationary point process,
Equation (\ref{cffpan}) gives
\begin{eqnarray}
\cfpn [G] & = & 
{\cal E}_0 [m^n_\g 1_{G}], \quad \forall G,
\label {cffpan2}
\end{eqnarray}
with $m^n_\g$ the random variable $m^n_\g(\phi,0)$
(see Equation (\ref{phifn})) and 
$1_{G}$ the indicator function $1_{G}(\phi)$.
So taking $G=\n_0$ gives
\begin{equation}\label{eqmean1}
{\cal E}_0 [m^n_\g]=1,
\end{equation}
which shows that, ${\cal P}_0$ a.s.,
$m^n_\g(\phi)<\infty$. 
This in turn implies that, ${\cal P}$ a.s.,
for all $y\in \phi$, $m^n_\g(\phi,y)<\infty$.
\item 
Equation (\ref{eqmean1}) together with the Campbell-Mecke formula
imply that the intensity of $\Phi_\g^n$ 
is equal to that of $\Phi$, as already mentioned.
\item
Equation (\ref{cffpan2}) shows that $\cfpn$ is absolutely continuous
w.r.t. ${\cal P}_0$, with  
Radon-Nikodym derivative 
$$m^n_f:= m^n_f(\phi,0).$$
\end{enumerate}
\end{rem}

\begin{prop} 
For all simple point
processes ${\cal P}$, for all $n$ and $G$,
\begin{eqnarray}
\cfpn [G] =   {\cal E}_0 \left[
\frac {m^n_\g}{{\cal E}_0 [ m^n_\g\mid m^n_\g>0]}
1_{G} \mid m^n_\g>0\right].
\label {cffpan3}
\end{eqnarray}
\end{prop}
\begin{proof}
Equation (\ref{cffpan}) implies that
\begin{eqnarray*}
\cfpn [G] & = & 
{\cal E}_0 [m^n_\g 1_{G}]=
{\cal E}_0 [m^n_\g 1_{G} 1_{m^n_\g>0}],
\label {cffpan22}
\end{eqnarray*}
Taking $G=\n_0$ gives
\begin{eqnarray*}
{\cal P}_0 [m^n_\g>0]= \frac 1
{{\cal E}_0 [ m^n_\g\mid m^n_\g>0]}.
\label {cffpan33}
\end{eqnarray*}
Equation (\ref{cffpan3}) then follows immediately.
\end{proof}
\subsubsection{Definition and Existence
of Point-Map-Probabilities}

\begin{defi}
\label{deffond}
Let $\g$ be a \pomk{} and let $\cal P$ be
a stationary point process with
Palm distribution $\cp_0$.
Every element of the $\omega$-limit set (see (\ref{eq:omegalim}))  of $\cp_0$ 
(where limits are w.r.t. the topology of the convergence
in distribution of probability measures on $\n^0$ 
-- cf. Subsection~\ref{RandMeas})
under the action of $\{(\thetg^n)_*\}_{n\in \bbN}$
will be called a \emph{$\g$-probability of} $\cp_0$.
In particular, if the limit of the sequence
$((\thetg^n)_* \cp_0)_{n=1}^\infty=(\cfpn)_{n=1}^\infty$ exists,
it will be called {\em the} $\g$-probability of $\cp_0$
and denoted by $\cfpi$.

\end{defi}

Let $A_{\cp_0}$ denote the orbit of $\cp_0$.
The set of $\g$-probabilities of $\cp_0$ is hence the set of
all accumulation points of the closure $\mathrm{cl}(A_{\cp_0})$ of
$A_{\cp_0}$, or equivalently the elements of
$\m^1(\n^0)$ the neighborhoods of which contain infinitely
many elements of $A_{\cp_0}$ -- see the definitions
in Section \ref{pernot}.

\begin{rem}
In view of (\ref{cffpan2}), for all $\cal P$ simple,
the existence of a unique $\g$-probability $\cfpi$ 
is equivalent to
\begin{eqnarray}
\lim_{n\to \infty} \int_{\n_0} h(\phi) m^n_\g(\phi)  {\cal P}_0(\d \phi)
= \int_{\n_0} h(\phi) \cfpi(\d \phi),
\label {wconv}
\end{eqnarray}
for all bounded and continuous functions $h:\n_0\to \ir$.
\end{rem}

\begin{cor}
Let $\cq$ be a $f$-probability.
Let $I$ be the set defined in (\ref{eq:definf}). 
If for all positive integers $n$, $\phi\to\ind_{m_f^n(\phi)>0}$
is $\cq$-a.s. continuous, then $\cq[I]=1$,
\end{cor}
\begin{proof}
The statement is an immediate consequence of Lemma~\ref{lem:discont.meas.zero}.
\end{proof}
The relative compactness of $A_{\cp_0}$
(and the existence of $\g$-probabilities) is not granted in general. 
The next lemmas give conditions
for this relatively compactness to hold.
From Lemma 4.5. in \cite{Ka76}, one gets:

\begin{lem}\label{kalen}
A necessary and sufficient condition for the set
$A_{\cp_0}$ to be relatively compact in $\m^1(\n^0(\R^d))$ is that
for all bounded Borel subsets $B$ of $\sp$,
\eq{\label{sufcon}
\lim_{r\to\infty}\limsup_{n\to \infty} \cfpn[\phi\in \n^0 \mbox{ s.t. }
\phi(B)>r]=0. }
\end{lem}

Examples of \pomk{} and point process pairs where
the last relative compactness property does
not hold are provided in Subsection \ref{CEPM}.
On stationary point processes,
all the \pomk{}s discussed in Section \ref{sr0} satisfy
this relative compactness property. 
For the periodic cases
(e.g. $c$, $n$ and $g$ on Poisson point processes), the result follows
from Proposition \ref{prop25} below, whereas for the strip \pomk{} $s$, the
proof is given in Subsection \ref{Markovian}. 

The \pomk{} $f$ will be said to have
{\em finite orbits} on the stationary point process 
$(\Phi,\p)$ if for $\ps\fp$-almost all $\phi$, 
$\{\thetg^n(\phi)\}_{n\in \bbN}$ is finite.
\begin{prop}
\label{prop25}
If $f$ has finite orbits
on the stationary point process $(\Phi,\p)$,
then the set $A_{\cp_0}$ is relatively compact.
\end{prop}
\begin{proof}
For all bounded Borel subsets $B$ of $\sp$ and $\phi\in\n^0$, let 
$$R_B(\phi):=\max_{n=0}^\infty\lg (\thetg^n\phi)(B)\rg.$$ 
Since $f$ has finite orbits, the RHS is the maximum over finite number of terms and hence $R_B$ is well-defined and finite. Clearly $R_B(\phi)\geq R_B(\thetg \phi)$ and therefore the distribution of the random variable $R_B$ under $\cfpn$ is stochastically decreasing w.r.t. $n$. Hence 
\begin{eqnarray*}
&&\hspace{-1cm}\lim_{r\to\infty}\limsup_{n\to \infty} \cfpn[\phi\in \n^0 \mbox{ s.t. } \phi(B)>r] \\
&\leq&\lim_{r\to\infty}\limsup_{n\to \infty} \cfpn[\phi\in \n^0 \mbox{ s.t. } R_B(\phi)>r]\\
&\leq&\lim_{r\to\infty}\cfp[\phi\in \n^0 \mbox{ s.t. } R_B(\phi)>r]=0.
\end{eqnarray*}
\end{proof}

\begin{rem}
It is easy to check that the following statements are equivalent.
(1) $f$ has finite orbits; (2) $f$  is periodic; and (3)
for $\ps\p$-almost all $\phi$, 
for all $x\in\phi$,
$\{F^n(\phi,x)\}_{n\in \bbN}$ has finitely many different points.
\end{rem}
So, for instance, for
the directional \pomk{} on the random geometric graph $g$,
$A_{\cp_0}$ is relatively compact as this \pomk{} is 1-periodic. 

\subsection{On Palm and \PoMk{}-Probabilities}
This subsection is focused
on the relation between Palm probabilities and \pomk{}-probabilities.
Throughout the subsection, $f$ is a \pomk{}, and $(\Phi,\p)$
is a simple and stationary point process with non-degenerate intensity.
The distribution of $\Phi$ is denoted by $\cal P$ and
its Palm probability by ${\cal P}_0$.

\subsubsection{Conditional Interpretation of the \PoMk{}-Probability}

The next theorem, which immediately follows from
Equation (\ref{cffpan3}), gives
a conditional definition of the $f$-probability from ${\cal P}_0$ :
\begin{thm}
\label{thmloc}
Let $\cal P$ be a simple stationary point process on 
$\sp$.
For all $n\in \mathbb N$ and $\phi\in \n$,
let $m^n_\g(\phi):=m^n_\g(\phi,0)$.
For all $n$, $m^n_\g(\phi)$ is ${\cal P}_0$ a.s. finite.
If there exists a unique $\g$-probability $\cfpi$ for $\cal P$,
then, for all $G$ such that $\cfpn [G]$ tends to $\cfpi[G]$
as $n$ tends to infinity, one has
\begin{eqnarray}
\cfpi [G]
 =  \lim_{n\to \infty}
{\cal E}_0 \left[
\frac {m^n_\g}{{\cal E}_0 [ m^n_\g\mid m^n_\g>0]}
1_{G}
\mid m^n_\g>0\right].
\label {cffpan4}
\end{eqnarray}
\end{thm}
Notice that, 
in addition to the conditioning, there
is a Radon-Nikodym derivative (w.r.t.
${\cal P}_0[\cdot\mid m^n_\g >0$]) equal to
$m^n_\g(\phi)/{\cal E}_0 [ m^n_\g\mid m^n_\g>0].$

\subsubsection{The Periodic Case}
Below, for all stationary point processes $(\Psi,\p)$ defined on
$(\Omega, \cal F)$ with a positive intensity,
$\p_\Psi$ denotes the Palm probability w.r.t. $\Psi$
on $(\Omega, \cal F)$.

\begin{lem}\label{finite.trap}
If $f$ is 1-periodic on the (simple) stationary point
process  $(\Phi,\p)$, then a.s., for all $x\in\Phi$, $\lim_n \Phi^n_f(\{x\})$
exists and is finite. 
\end{lem}

\begin{proof} 
If  $x$ is a trap of $\Phi $, i.e. $F(\Phi,x)=x$,  then
$(\Phi^f_n(x))_{n=1}^\infty$ is non-decreasing in $n$. 
Let $\Psi$ be the thinning of $\Phi$ to traps of $\Phi$
for which the above limit is not finite. 
The compatibility of $F$ implies that $(\Psi,\p)$ 
is a stationary point process. 
If $B$ is the unit box in $\sp$ and  $K$ is a positive integer,
for $n$  large enough, one has 
\eqn{\lambda_\Phi =
\int_\Omega\Phi^n_f(B) \p(\d \omega)\geq
\int_\Omega\Phi^n_f(\Psi\cap B) \p(\d \omega)
\geq \int_\Omega K\Psi(B)\p(\d \omega) =K\lambda_\Psi,}
where $\lambda_\Phi$ and $\lambda_\Psi$ denote the 
intensities of the point processes.
Therefore $\lambda_\Psi\leq\lambda_\Phi/K$, 
which proves that $\lambda_\Psi=0$.
Hence, a.s., at the traps of $\Phi$, the limit exists and is finite.
Given this, it is easy to verify that if $y\in\Phi$ is not a trap,
for $n$ large enough, $\Phi^n_f(y)=0$ and hence the limit
exists for all points of $\Phi$. 
\end{proof}

When $\lim_n \phi^n_f$ exists and is a counting measure,
it is denoted it by $\Phi^\infty_f$.
Hence, in the 1-periodic case, $(\Phi^\infty_f,\p)$ is
well defined and a non-degenerate stationary point process.

\begin{thm}\label{thmcyc1}
If $f$ is 1-periodic on $(\Phi,\p)$,
then the $\g$-probability $\cfpi$ of
$\cp_0=\Phi_* \p_\Phi$
exists and is given by
\eq{\cfpi= \Phi_*\  \p_{\Phi_f^\infty}.}
Let $m^\infty_\g(\Phi)$ denote the multiplicity
of the origin under $\p_{\Phi_f^\infty}$.
Then $\cfpi$ is absolutely continuous
with respect to $\cp_0$, with    
\eq{\label{tleq1}\frac{\d\cfpi}{\d\cp_0}(\phi)=
m^\infty_\g(\phi).}
In addition, $\cfpi= \Thetg\cfpi$.
\end{thm}

\begin{proof}
In the 1-periodic case, for all bounded  Borel sets $B$,
$\Phi_f^n(B)$ a.s. coincides with $\Phi_f^\infty(B)$ for $n$ large enough,
so that by letting $n$ to infinity in (\ref{cffpan}),
one gets that for all $G\in {\cal N}$, the limit
\eq{\label {ffpani}
\lim_n
\cfpn[G] = 
\frac 1 {\lambda \no B} 
\int_{\n}
\int_{B}
\mathbf {1} \{\theta_t\phi\in G\}
\phi_f^\infty(\d t)
\cp [\d\phi]}
exists.
Since $\phi_f^\infty$ is a stationary point process with
the same intensity as the original point process
(because of the conservation of intensity),
$\fpi$ is  the distribution
of $\Phi$ with respect to the Palm
distribution of $\Phi_f^\infty$ indeed. In addition, 
for all $H\in {\cal F}$
\begin{eqnarray*}
\p_{\Phi_f^\infty}[H]
&=&\frac 1 {\lambda \no B} 
\int_\Om \int_B
\mathbf {1} \{\theta_t\om\in H\}
\Phi_f^\infty(\om,\d t)
\p [\d\om]\\
&=&\frac 1 {\lambda \no B}
\int_\Om
\int_B
\Phi_f^\infty(\om,\{t\})
\mathbf {1} \{\theta_t\om\in H\}
\Phi(\om,\d t)
\p [\d\om]\\
&=&\frac 1 {\lambda \no B}
\int_\Om
\int_B
\Phi_f^\infty(\theta_t\om,\{0\})
\mathbf {1} \{\theta_t\om\in H\}
\Phi(\om,\d t)
\p [\d\om]\\
&=&\e_\Phi\lk\Phi_f^\infty(\{0\})\mathbf {1}_H(\Phi)\rk=
\e_\Phi\lk m_f^\infty(\Phi)\mathbf {1}_H(\Phi)\rk,
\end{eqnarray*}
where the second equality stems from the facts that 
$\overline{\Phi_f^\infty}\subset\Phi$ and that $\Phi$ is simple.
This proves (\ref{tleq1}) when $H=\Phi^{-1} G$. 
Finally since $f$ is 1-periodic, $\cfpi$-almost surely,  $\g\equiv0$ 
which proves that $\cfpi$ is invariant under the action of $\Thetg$ .
\end{proof}

The \pomk{} $g$ provides an examples where
Theorem \ref{thmcyc1} holds.
See Subsection \ref{CHC}.
Note that similar statements hold in the $p$-periodic case.
In this case,  $f^p$
is 1-periodic on the point processes $\{(\Phi,\cfpn)\}_{n=0}^{p-1}$,
and hence there exists at most $p$ \pomk{}-probabilities.
Details on this fact are omitted.

\subsubsection{The Evaporation Case}
The next theorem shows that
in contrast to Theorem~\ref{thmcyc1} where the $f$-probability is 
absolutely continuous with respect to the Palm probability,
there are cases where the $f$-probability and the Palm probability
are singular. This theorem is based on
the following lemma: 

\begin{lem}\label{infimage}
Let $I$ be the set defined in (\ref{eq:III}).
If $\cq$ is a probability distribution on $\n^0$
which satisfies $\Thetg\cq=\cq$, then $\cq[I]=1$. 
In this case, $\cq$-almost surely, there exists a bi-infinite 
path (which can be a periodic orbit) which passes through the origin;
i.e., $\{y_i=y_i(\phi)\}_{i\in\Z}$ is such that $y_0=0$
and $ \f(\phi,y_i)=y_{i+1}$.
\end{lem}
\begin{proof}
Let $M_n:=\{\phi\in\n^0; \f^{-n}(\phi,0)=\emptyset\}$,
where  $ \f^n(\phi,\cdot)$ is
defined in Subsection  \ref{sec:cpoms-ite}.
It is sufficient to show that, for all $n>0$, $\cq[M_n]=0$.
But the invariance of $\cq$ under the action of $\Thetg$ gives
\begin{eqnarray*}
\cq[M_n] & = & \Thetg^n\cq[M_n]=\cq\lk(\thetg)^{-n}M_n\rk
\\ & = & \cq\lk
\{\phi\in\n^0;\f^{-n}(\phi,\f^n(\phi,0))=\emptyset\}\rk=0.
\end{eqnarray*}
The proof of the second statement is clear if the orbit of $\phi$ is periodic under the action of $\thetg$ and if not, it is  an immediate
consequence of K\"onig's infinity lemma \cite{Ko90}.
\end{proof}

\begin{thm}
\label{theevasin}
If the stationary point process $(\Phi,\p)$ evaporates under the action of $f$, and if the $\g$-probability $\cfpi$
of 
$\cp_0=\Phi_* \p^0_\Phi$
exists and satisfies $\cfpi= \Thetg\cfpi$,
then $\cfpi$ is singular with respect to $\cp_0$. 
\end{thm}
\begin{proof}
The result is obtained when
combining Lemmas~\ref{infimage} and~\ref{eveq}.
\end{proof}
It is shown in Subsection \ref{Markovian} 
that the assumptions of Theorem \ref{theevasin} are satisfied 
by the strip \pomk{} $s$ 
on Poisson point processes in ${\mathbb R}^2$.

\begin{rem}
The case with evaporation is that where the conditioning representation
given in Equation (\ref{cffpan4}) is w.r.t. an event whose probability
w.r.t. ${\cal P}_0$ tends to 0 as $n$ tends to infinity.
\end{rem}
\begin{rem}
The singularity property established in Theorem \ref{theevasin}
can be completed by the following observation:
under the assumptions of this theorem,
there is no finite and measurable $U=U(\phi)\in \phi$ (resp. 
$V=V(\phi)\in \phi$) such that
$\cfpi = (\theta_U)_* {\cal P}$,
(resp. $\cfpi = (\theta_V)_* {\cal P}_0$),
i.e., there is no {\em shift-coupling} giving $\cfpi$ as function
of ${\cal P}$ (resp. ${\cal P}_0$). The proof is by
contradiction: evaporation implies that
${\cal P}$ (resp. ${\cal P}_0$) a.s.,
$\theta_x \phi\notin I$ for all $x\in \phi$.
But this together with $\cfpi = (\theta_U)_* {\cal P}$
(resp. $\cfpi = (\theta_V)_* {\cal P}_0$)
imply that $\cfpi[I]=0$, which contradicts the fact that,
under the assumptions of Theorem \ref{theevasin}, $\cfpi[I]=1$.
\end{rem}

\subsection{Mecke's Point-Stationarity Revisited}
\subsubsection{Mecke's Invariant Measure Equation}
Consider the following \pomk{} invariant measure equation
\eq{\label{inveq}\Thetg\cq=\cq,}
where the unknown is $\cq\in M^1(\n^0)$.
From Mecke's point stationarity Theorem \ref{point.sta},
if $\thetg$ (or equivalently $F$) is bijective,
then the Palm probability $\cp_0$ of any simple stationary point process
solves (\ref{inveq}). From Theorem \ref{thmcyc1},
if $f$ is 1-periodic on $(\Phi,\p)$, then the $f$-probability 
of $\Phi$ exists and from the last statement of this theorem,
it satisfies (\ref{inveq}).
More precisely, a solution to (\ref{inveq}) was built
from the Palm probability ${\cal P}_0$ of $\Phi$ by Equation (\ref{tleq1}).

Equation (\ref{inveq}) will be referred to as 
{\em Mecke's invariant measure equation}.
The bijective case shows that the solution of (\ref{inveq}) 
is not unique in general (all Palm probabilities are solution).

A natural question is whether one can
construct a solution of (\ref{inveq}) from the Palm probability
of a stationary point process beyond
the bijective and the 1-periodic cases, for instance when
$\Phi$ evaporates under the action of $f$.

Consider the Ces\`aro sums
\eq{\label{cesdef} \ccfpn:=\frac 1n \sum_{i=0}^{n-1}\cp_0^{\g,i}, \quad n\in 
\bbN.}
When the limit of $\ccfpn$ as $n$ tends to infinity exists
(w.r.t. the topology of $\m^1(\n^0)$), let
\eq{\label{cesdefinfty} \ccfpi:=\lim_{n\to \infty}\frac
1n \sum_{i=0}^{n-1}\cp_0^{\g,i}.}
In general,  $\ccfpi$ is not a $\g$-probability.  

\begin{thm}\label{continv}
Assume there exists a subsequence  $(\widetilde\cp_0^{\g,n_i})_{i=1}^\infty$ which converges to a probability measure $\ccfpi$. If $\Thetg$
is continuous at $\ccfpi $, then $\ccfpi$ solves Mecke's invariant
measure equation (\ref{inveq}).
\end{thm}

\begin{proof}
From (\ref{cpntocpn+1}),
\begin{eqnarray}\label{inv.thet.ces}
\Thetg\ccfpn-\ccfpn&=&\hspace{-.2cm}\frac 1n\lp\sum_{i=0}^{n-1}\Thetg\cp_0^{\g,i}-\sum_{i=0}^{n-1}\cp_0^{\g,i}\rp\nonumber\\
&=&\hspace{-.2cm}\frac 1n\lp\sum_{i=1}^n\cp_0^{\g,i}-\sum_{i=0}^{n-1}\cp_0^{\g,i}\rp=\frac 1n\lp\cfpn-\cfp\rp.
\end{eqnarray}
Therefore, if the subsequence  $(\widetilde\cp_0^{f,n_i})_{i=1}^\infty$  converges in distribution w.r.t. the vague topology to a probability measure $\ccfpi$, then  (\ref{inv.thet.ces}) implies that the sequence $(\Thetg\widetilde\cp_0^{f,n_i})_{i=1}^\infty$ converges to $\ccfpi$ as well. Now the continuity of $\Thetg$ at $\ccfpi$ implies that $(\Thetg\widetilde\cp_0^{f,n_i})_{i=1}^\infty$ converges to $\Thetg\ccfpi$ and therefore $\Thetg\ccfpi=\ccfpi.$
\end{proof}

\begin{rem}
Here are some comments on the last theorem:
\begin{enumerate}
\item A sufficient condition for the existence of
a converging
subsequence in Theorem \ref{continv} is the relative compactness
condition of Lemma \ref{kalen}.
\item When the sequence $(\cfpn)_{n=1}^\infty$ 
converges to $\cfpi$, then $(\ccfpn)_{n=1}^\infty$
converges to $\cfpi$ too, and hence 
Theorem \ref{continv} implies the invariance of
the $\g$-probability $\cfpi$ under the action of $\Thetg$,
whenever  $\Thetg$ has the required continuity.
\item
If instead of $(\ccfpn)_{n=1}^\infty$, $(\cfpn)_{n=1}^\infty$
has convergent subsequences with different limits, i.e., if the set of
$\g$-probabilities is not a singleton,
then none of the $\g$-probabilities satisfies (\ref{inveq}).
However, it follows from Lemma \ref{lemcherche} in the appendix
that if $(\theta_\g)_*$ is continuous, and if $(\cfpn)_{n=1}^\infty$ is
relatively compact, then the set of $\g$-probabilities of $\cp_0$ is
compact, non empty and $(\theta_\g)_*$-invariant.
\item
The conditions listed in Theorem \ref{continv} are all
required.
There exist \pomk{}s $f$ such that $(\ccfpn)_{n=1}^\infty$
has no convergent subsequence (see Subsection \ref{CEPM});
there also  exist \pomk{}s $f$ such that
$(\cfpn)_{n=1}^\infty$ is convergent, but $\Thetg$
is not continuous at the limit and $\cfpi$ is not
invariant under the action of $\Thetg$ (see Subsection \ref{sr}).
The use of Ces\`aro limits is required too as there
exist \pomk{}s $f$ such that
$(\cfpn)_{n=1}^\infty$ is not convergent,
whereas $(\ccfpn)_{n=1}^\infty$ converges to a limit
which satisfies (\ref{inveq}) (see Subsection \ref{nonuniqueconv.sub.}).
\end{enumerate}
\end{rem}
\subsubsection{Continuity Condition}
In  case of  existence of $\ccfpi$,
Theorem~\ref{continv} gives a sufficient condition for $\ccfpi$
to solve (\ref{inveq}); however since $\ccfpi$ lives 
in the space of probability measures on counting measures,
the verification of the continuity of $\Thetg$ at $\ccfpi$ 
can be difficult. The following  propositions give more handy
tools to verify the continuity criterion.

\begin{prop}\label{f.cont}
If $\thetg$ is $\ccfpi$-a.s. continuous,
then $\Thetg$ is continuous at $\ccfpi$.  
\end{prop}
\begin{proof}
The proof is an immediate consequence of
Proposition~\ref{cont.m.cont} in the appendix,
as the space $\n(\sp)$ is a Polish space. 
\end{proof}

\begin{prop}\label{s.cont}
If $\g$ is $\ccfpi$-almost surely continuous,
then $\Thetg$ is $\ccfpi$-continuous.  
\end{prop}

\begin{proof}
One can verify that  $\theta:\sp\times\n\to\n$ defined by $\theta(t,\phi)=\theta_t\phi$ is continuous. Also $h:\n^0\to\sp\times\n$ defined by $h(\phi)=(f(\phi),\phi)$ is continuous at continuity points of $f$ in $\n^0$. Hence $\thetg=\theta\circ h$ is continuous at continuity points of $\g$. 
\end{proof}
The converse of the statement of Proposition~\ref{s.cont} does not hold in 
general (see Subsection \ref{CEPM}). 
Combining the last propositions and Theorem \ref{continv} gives:

\begin{cor}
If the limit $\ccfpi$ defined in (\ref{cesdefinfty}) exists and
if in addition $\g$ is $\ccfpi$-almost surely continuous,
then $\Thetg$ is continuous at $\ccfpi$, and $\ccfpi$ then
solves Mecke's invariant measure equation (\ref{inveq}).
\end{cor}

In Theorem \ref{continv} and the last propositions,
the continuity of the mapping $\Thetg$ 
is required at some specific point only.
The continuity of $f$
is a stronger requirement which does not hold for most interesting cases as
shown by the following proposition
(see Appendix \ref{seccontpos} for a proof).

\begin{prop}\label{contpos}
For $d\geq 2$, there is no continuous \pomk{} on the whole
space $\n^0$ other than the point-map of the identity \pom{};
i.e., the \pomk{} which maps all $\phi\in\n^0$ to the origin.  
\end{prop} 

\subsubsection{Regeneration}
In certain cases, the existence of $\cfpi$ can be established using the theory of regenerative processes \cite{As03}. This method can be used  when the point process satisfies the \emph{strong Markov property} such as Poisson point processes \cite{Mo05}.

Assume $\g$ is a fixed \pomk{} and $(\Phi,\fp)$ is the Palm version of a stationary point process. 
For $n\geq 0$
let 
\begin{equation}
\label{eq:X}
\xfai=\xfai(\g,\Phi)=f^n(\Phi)\in \sp,
\end{equation}
where $f^n$ is defined in (\ref{nthpom}). Note that $\fp$-almost surely, $\Phi\in\n^0$ and hence $X_n$ is well defined. 
Finally, denote $\theta_{\xfai}\Phi$ by $\Pn$
(this point process should not be confused with
$\Phi^n_{f}$ defined in Subsection \ref{sec:cpoms-mul})
and by $\Prn$ the restriction of $\Pn$ to the sphere of radius $r$
centered at the origin. Using this notation,
Lemma~\ref{convfrom0} gives $(\theta_{\xfai})_*\cfp=\cfpn$ or 
equivalently $(\Pn)_*\fp=\cfpn$.

The following theorem leverages  classical results in the theory of regenerative processes.  

\begin{thm}\label{regenver}
If, for all $r>0$, there exists a strictly increasing sequence of non-lattice integer-valued random variables
$(\eta_i)_{i=1}^\infty$, which may depend on $r$, such that 
\begin{enumerate}
\item $(\eta_{i+1}-\eta_i)_{i=1}^\infty$ is a sequence of i.i.d. random variables with finite mean,
\item the sequence $Y_i:=(\Phi_{\eta_i}^r,\Phi_{\eta_i+1}^r, \ldots, \Phi_{\eta_{i+1}-1}^r)$ is an i.i.d. sequence and $Y_{i+1}$ is independent of $\eta_1,\ldots,\eta_i$, 
\end{enumerate}
then the $\g$-probability $\cfpi$ exists and, for all bounded and measurable
functions $h$ and for $\cfp$-almost all $\phi$, 
\begin{equation}\label{regeneq}
 \lim_{k\to\infty} \frac 1 k \sum_{n=0}^{k-1} h(\thetg^n\phi) = 
\int_{\n^0} h(\psi) \cfpi(\d\psi). 
\end{equation}
If in addition, for all $n$, $\g$ is $\cfpn$-almost  surely continuous, then $\cfpi$ is invariant under the action of $\Thetg$ and $\thetg$ is ergodic on $(\n^0,\N^0,\cfpi)$.    
\end{thm}

\begin{proof}
In order to prove the weak convergence of $\cfpn$ to $\cfpi$, it is sufficient to show the convergence in all balls of integer radius $r$ around the origin. Note that $\cfpn$ is the distribution of $\Pn$ and hence, to prove the existence of $\cfpi$, it is sufficient to prove the convergence of the distribution of $\Prn$ for all $r\in \Bbb N$.

Note that the sequence $(\eta_i)_{i=1}^\infty$ forms a sequence of regenerative times for the configurations  in $B_r(0)$. Since $\n^0$ is metrizable (c.f. \cite{As03}, Theorem~B.1.2), the distribution of $\Prn$ converges to a distribution $\cp_{0,r}^\g$ on configurations of points in $B_r(0)$ satisfying 
\begin{equation}\label{regeneq2}
\frac 1{\cfe[\eta_2-\eta_1]}\cfe\lk\sum_{n=\eta_1}^{\eta_2-1} h(\Prn)\rk = 
\int_{\n^0} h(\psi\cap B_r(0)) \cp_{0,r}^{f}\d(\psi\cap B_r(0)),
\end{equation}
for all $h:\n_0\to \ir^+$.
Since the distributions $(\cp_{0,r}^f)_{r=1}^\infty$ are the limits  of $(\Prn)_{r=1}^\infty$, they satisfy the consistency condition of Kolmogorov's extension theorem and therefore there exists a probability distribution $\cfpi$ on $\n^0$ having $\cp_{0,r}^\g$ as the distribution of its restriction to $B_r(0)$.
This proves the existence of the $\g$-probability.

The left-hand side of (\ref{regeneq2}) can be replaced by
an ergodic average
(c.f. \cite{As03} Theorem B.3.1); i.e., for all $r\in\Bbb N$,
for $\cfp$-almost all $\phi\in\n^0$,
\begin{eqnarray*}\lim_{k\to\infty}\frac 1{k}\sum_{n=0}^{k-1} h(\thetg^n\phi\cap B_r(0)) &=& 
\int_{\n^0} h(\psi\cap B_r(0)) \cp_0^{\g,r}\d(\psi\cap B_r(0))\\
&=& 
\int_{\n^0} h(\psi\cap B_r(0)) \cfpi(\d\psi). 
\end{eqnarray*}
Finally $r$ varies in the integers and hence the last equation gives (\ref{regeneq}), for $\cfp$-almost all $\phi$.

By defining $h$ as the continuity indicator of $\g$, the $\cfpn$-almost sure continuity of $\g$ and (\ref{regeneq}) give its $\cfpi$-almost sure continuity and hence that of $\Thetg$ at $\cfpi$. Therefore $\cfpi$ is invariant under the action of $\Thetg$. Also ergodicity is clear from  regeneration. 
\end{proof}

The main technical difficulty for using Theorem~\ref{regenver}
consists in finding an appropriate sequence $(\eta_i)_{i=1}^\infty$.
Proposition~\ref{fstpa} below leverages the strong Markov property
of Poisson point processes to find appropriate sequences and prove
the existence of the \pomk{} probability for the \pomk{} $s$ for
homogeneous Poisson point processes.
Proposition~\ref{fuapa} uses the same approach to show that
the same holds true for the directional \pomk{} $d_\alpha$.
Other examples can be found
in Section \ref{proofs}.

\section{More on Examples}\label{proofs}

\subsection{Strip  \PoM{}} \label{sr}
Let $\cp_0$ denote the Palm distribution of the homogeneous Poisson
point process on $\R^2$. It follows from results in \cite{FeLaTh04} 
(in Theorem 3.1. of this reference,
the authors proved that the graph of this \pom{} has 
finite branches, which is equivalent to evaporation)
that $\cp_0$ evaporates  under the action of the strip \pomk{} $s$. 
It is also shown in Proposition~\ref{fstpa} below that $\cp_0$,
admits a unique $s$-probability which 
satisfies the continuity requirements of Theorem \ref{continv}.

This \pom{} also allows one to illustrate 
the need of the continuity property in Theorem \ref{continv}.
Consider the setup of Proposition~\ref{fstpa}.
For all $\phi\in\n^0$ such that the origin has infinitely many pre-images,
change the definition of the \pomk{} $s$ as follows: it is
now the closest point on the
right half plane which has no other point of of $\phi$ in the
ball of radius $1$ around it.
Due to evaporation, this changes the definition of $s$ on a set
of measure zero under ${\cal P}_0^{s,n}$, for all $n\in\Bbb N$, and hence,
the sequence $(\cp_0^{s,n})_{n=1}^\infty$
is again converging
to the same limit as that defined in the proof of Proposition~\ref{fstpa}.
But under the action of the new $s$, $(\theta_{s})_*\cp_0^{s}$
is not equal to $\cp_0^{s}$ due to the facts 
that (i) $0$ has infinitely many pre-images $\cp_0^{s}$-a.s. and (ii) 
there is no point of the point process in the ball of radius $1$.
This does not agree with the fact that, in the right half plane,
the distribution of $\cp_0^{s}$ is a Poisson point process (see
the proof of Proposition~\ref{fstpa}). Hence,
$\cp_0^{s}$ is not invariant under the action of $(\theta_{s})_*$.
\subsection{Directional \PoM{}}
\label{dr} 
The directional \pom{} was introduced in \cite{BaBo07}.
Let $e_1$ be the first coordinate unit vector
The \emph{directional \pomk{} $d$}
maps the origin to the nearest
point in the right half-space, defined by $e_1$,
i.e., for all $\phi\in\n^0$,
\eq{\label {drd}d(\phi):=\argmin\{\norm{y};y\in  \phi,y\cdot e_1>0\}.}
The associated \pom{} will be denoted by $D$. 

The directional \pomk{} on $\R^2$ with deviation limit $\alpha$,
$d_{\al}$, is similar to $d$, except that the
point $y$ is chosen in the cone with angle $2\al$ and central
direction $e_1$ rather than in a half plane; i.e., 
\eq{\label {drda}
d_{\al}(\phi,x):=\argmin\{\norm{y};y\in  \phi,\frac{y}{\norm{y}}\cdot u>\cos \al\}.}
When $\al=\frac\pi 2$ one has  $d_{\al}=d$.
Its \pom{} is denoted by $D_{\al}$.

When $\al<\pi/2$, it can be shown that
the homogeneous Poisson point process on $\R^2$ evaporates under
the action of $d_\al$, and from Proposition~\ref{fuapa} below,
it admits a unique $d_\al$-probability which 
satisfies the continuity requirements of Theorem \ref{continv}.

\subsection{Regeneration} \label{Markovian}
This subsection is focused on the existence of \pomk{}-probabilities for
point-maps defined on Poisson point processes.
It is based on Theorem~\ref{regenver} and is illustrated by  two examples. 

\begin{prop}\label{fstpa}
If $s$ is the strip \pomk{}, and
$(\Phi,\p)$ is a homogeneous Poisson point process 
in the plane with distribution $\cp$, then
the $s$-probability exists and is given by (\ref{regeneq}). 
In addition, for all $n$,  $s$ is $\cp_0^{s,n}$-almost surely continuous.
Therefore the action of $(\theta_{s})_*$ preserves $\cp_0^{s}$ and is ergodic.
\end{prop}

\begin{proof}

The random vector, $X_1=X_1(\Phi)$ defined in Equation (\ref{eq:X}),
depends only on the points of $\Phi$
which belong to the rectangle $\sfa_0(\Phi)=[0,x_1]\times[-1,1]$,
where $x_1$ is the first coordinate of the left most point of
$\Phi\cap T(0)$. It is easy to verify that $\sfa_0(\Phi)$
is a stopping set (c.f. \cite{Mo05} and \cite{Zu06}). 
Let $\sfai(\Phi)$ be the rectangle which is needed to determine the
image of the origin in  $\theta_{\xfai}\Phi$ under the action of $\fst$.
Let $\sfai+\xfai$ be the translation of the set $\sfai$
by the vector $\xfai$
Then it is clear that  
\eq{\label{unv.dif}\ufak=\bigcup_{n=0}^k \lp\sfai+\xfai\rp}
is also a stopping set.
As a consequence, the strong Markov property of Poisson point
process (c.f. \cite{Zu06}) implies, given $X_0,\ldots,\xfai$, 
the point process on the right half-plane of $\xfai$ is distributed
as the original Poisson point process. Let 
\eqn{p_n=\pi_1(X_{n+1}-\xfai),}
where $\pi_1$ is the projection on the first coordinate. 
Since $\cp_0^{s,n}$, restricted to the right half-plane,
is the distribution of a Poisson point process and since
the sequence $(p_n)_{n=1}^\infty$ depends only on the
configuration of points in the right half-plane,
$(p_n)_{n=1}^\infty$ is a sequence of i.i.d. exponential
random variables with parameter $2\lambda$, where $\lambda$
is the intensity of the point process. 
Also if $\eta_i$ is the integer $n$ such that, for the $i$-th time,
$p_{n}$ is larger than $2r$, then the sequence 
$(\eta_i)_{i=1}^\infty$ forms a sequence of regenerative
times for configuration of points in $B_r(0)$.
Combining this with the distribution of $p_n$ gives that
$(\eta_i)_{i=1}^\infty$ satisfies the required
conditions in Theorem~\ref{regenver}. 

Finally, consider the discontinuity points of $s$. Let $\phi\in\n^1$ with
$s(\phi)=x=(x_1,x_2)$. It is shown below that if
$\phi$ is a discontinuity point of
$s$, then either $x$ lies on the boundary of $T(0)$ or there is a point
of $\phi$ other than the origin and $x$ which lies on the perimeter of
the rectangle $[0,x_1]\times[-1,1]$. This proves that, for all $n$,
the discontinuity points of $s$ are of $\cfpn$-zero measure. To prove the
continuity claim, assume that $\phi$ satisfies none of the above condition.
Hence there exists $\epsilon>0$ such that, $x_1>\epsilon$,
$x_2\in[-(1-\epsilon),1-\epsilon]$ and there is no other point of $\phi$
in $[-\epsilon,x_1+2\epsilon],[-1-\epsilon,1+\epsilon]$. Therefore, for
$\psi\in\n^0$ close enough to $\phi$ in the vague topology, there is a
point $y=(y_1,y_2)\in\psi$ in an $\epsilon$-neighborhood of $x$, which gives
$y\in(0,x_1+\epsilon)\times(-1,1)$ and since there is no  point of
$\psi$ other than $0$ and $y$ in $[0,x_1+\epsilon]\times[-1,1]$,
$s(\psi)=y$, which proves the claim.

Therefore all conditions of Theorem~\ref{regenver} are satisfied,
which proves the proposition.
\end{proof}

Note that the proof shows that the distribution ${\cal P}^s_0$
on the right half-plane is homogeneous Poisson with the original intensity. 

\begin{prop}\label{fuapa}
Let $d_{\al}$ be the directional \pomk{} defined in Subsection \ref{dr}
with $\al<\pi/2$. Under the assumptions of Proposition~\ref{fstpa}, 
the $d_{\al}$-probability exists and is given by (\ref{regeneq}).
In addition, for all $n$, $d_{\al}$ is $\cp_0^{d_{\al},n}$-almost surely
continuous and hence the action of $(\theta_{d_{\al}})_*$
preserves $\cp_0^{d_{\al}}$ and is ergodic.
\end{prop}

\begin{proof}
The proof is similar to that of Proposition~\ref{fstpa}, but more subtle.
It uses the same notation as that of Theorem~\ref{regenver}.

Let $C^{\al}$ denote the cone with angle $2\al$, central direction $e_1$,
and apex at the origin.
Let $X_1(\phi)$ be the point of $C^{\al}\cap\phi$ which is
the closest to the origin (other than the origin itself).
Let $\cfu_0(\phi)$ be the closed subset of $C^{\al}$ consisting
of all points of $C^{\al}$ which are not
farther to the origin than $X_1(\phi)$. This set will
be referred to as a bounded cone below.
One may verify that $\cfu_0(\phi)$ is a stopping set and that
$X_1$ is determined by $\cfu_0$. 
Let $\cfui(\phi)$ be the closed bounded cone which
is needed to determine the image of the origin in
$\theta_{\xfui}\phi$ under the action of $\fua$.
It is easy to verify that  
\eq{\ufuk=\bigcup_{n=0}^k \lp\cfui+\xfui\rp,}
is also a stopping set. It is a  simple geometric fact that  
\eq{\label{empty.inetr}\ufu_{n-1}\cap C^{\pi/2-\al}(\xfui)=\{\xfui\},}
and as a consequence, given $\ufu_{n-1}$, the point process in 
$C^{\pi/2-\al}+\xfui$ is distributed as the original point process.
This fact together with the facts that $\ufai$ is a stopping set
and  $\cfui$ has no point of the point process other than $\xfui$
and $X_{n+1}$, give that, in the $n$-th step, with probability at least 
$\min\{1,(\pi/2-\al)/(\al)\}$,
$X_{n+1}$ is in $C^{\pi/2-\al}(\xfui)$.
Let $\eta_i$ be the $i$-th time for which
$X_{n+1}\in C^{\pi/2-\al}(\xfui)$ and has a distance more than $2r$
from the edges of $C^{\pi/2-\al}(\xfui)$.
The Poisson distribution of points in $C^{\pi/2-\al}(\xfui)$
gives that the random variables $\eta_{i+1}-\eta_i$
are stochastically bounded by an exponential random variable
and hence they satisfy all requirements of Theorem~\ref{regenver}. 

As in the case of the strip \pom, it can be shown that if $d_\al$ is 
not continuous at $\phi\in\n^0$ then either there is no point in the
interior of $\cfu$ or there is a point on the perimeter of $\cfu_0(\phi)$.

Note that since $\ufa_{n-1}$ is a stopping set and
$(\cfu+\xfai)\cap\ufa_{n-1}$ has no point of the point process
other than $\xfui$, $X_{n+1}$ is distributed as in a Poisson point
process in $\cfui+\xfai$ given the fact that some parts contain no point.
Therefore since the discontinuities of $d_{\al}$ are of probability
zero under the Poisson distribution, they are of probability zero under
all $\cp_0^{d_{\al},n}$ and hence
Theorem~\ref{regenver} proves the statements of the proposition. 
\end{proof}

The statement of Proposition~\ref{fuapa} is also true in the case
 $\al=\pi/2$ and can be proved using ideas similar to those
in the proof for $\al<\pi/2$. However the technical details
of the proof in this case may hide the main idea and this case
is hence ignored in the proposition. 

\subsection{Condenser and Expander \PoM}
\label{CEPM}
Assume each point $x\in\phi$ is marked with  
\eqn{\nu_p(x)=\#(\phi\cap B_1(x)) \quad (\text{respectively }
\nu_m(x)=\sup\{r>0: \phi\cap B_r(x)=\{x\}\}),}
where $B_r(x)=\{y\in \R^2:\ \norm{x-y}<r\}$.
Note that $\nu_p(x)$ and $\nu_m(x)$ are always positive.
The \emph{condenser \pom} $P$ 
(respectively \emph{expander \pom} $M$) acts on 
counting measures
as follows: it goes from each point $x\in\phi$
to the closest point $y$ such that $\nu_p(y)\geq 2\nu_p(x)$ 
(respectively $\nu_m(y)\geq 2\nu_m(x)$).
It is easy to verify that both \pom s are compatible
and almost surely well-defined on the homogeneous Poisson point process.    

Poisson point processes evaporate under the action
of both \pom{}s $P$ and $M$.

The condenser \pomk{} provides 
an example where no $f$-probability exists.
Let $(\id,\cp)$ be the Poisson point process
with intensity one on $\R^2$ and let $p$ be the condenser \pomk{}. Clearly 
\eqn{\cp_0^{p,n}[\phi(B_1(0))>2^n]=1.}
Therefore the tightness criterion is not satisfied and thus
there is no convergent subsequence of $(\cp_0^{p,n})_{n=1}^\infty$. 

Similarly, the expander \pomk{} allows
one to show that there is no converse to Proposition \ref{s.cont}.
More precisely, $\theta_m$ is continuous
${\cal P}_0^m$-almost surely but the \pomk{} is ${\cal P}_0^m$-almost
surely discontinuous. Hence 
the converse of the statement of Proposition~\ref{s.cont}
does not hold in general.
Consider $m$ on the homogeneous Poisson point process.
One can verify that $({\cal P}_0^{m,n})_{n=1}^\infty$
converges to the probability measure concentrated on the
counting measure $\delta_0$ with a single point at the origin.
In this example, $\theta_m$ is ${\cal P}_0^m$-a.s. continuous.
This follows from the fact that when looking at the point process
in any bounded subset of $\sp$, it will be included in some ball
of radius $r$ around the origin and therefore
the configuration of points in it will be constant
(only one point at the origin) after finitely many application of $\theta_m$.
But the \pomk{} $m$ makes larger and larger steps and hence the sequence of
laws of $m$ under ${\cal P}_0^{m,n}$ diverges.
Hence $m$ is almost surely not continuous at the realization
$\delta_0$ on which ${\cal P}_0^{m}$ is concentrated.     

\subsection{Closest Hard Core \PoM}
\label{CHC}
By definition, the image of $x\in \phi$ by the
\emph{closest hard core \pom} $H$ is the closest point $y$
of $\phi$ (including $x$ itself) such that $\phi(B_1(y))=1$. 
Its \pomk{} will be denoted by $h$.

The \pomk{} $h$ is 1-periodic.
It provides an illustration of Theorem \ref{thmcyc1}.
Consider $h$ acting on a stationary Poisson point process
of intensity one in the plane.
For the simple counting measure $\phi$,
let $\Psi(\phi)$ denote sub-point process of $\phi$
made of points $y$ of $\phi$ such that $\phi(B_1(y))=1$.
If $\phi$ is chosen w.r.t. $\cp$, then $\Psi(\phi)$
is also a stationary point process.
Let $\cfq$ denote the Palm probability of $\Psi(\phi)$.
Then $\cfpi$ is  absolutely continuous w.r.t. $\cfq$ and its
Radon-Nikodym derivative at each $\Psi(\phi)\in\n^0$ is proportional 
to  the number of  points of $\phi$ in the Voronoi cell of the
origin in $\Psi(\phi)$.  

\subsection{Quadri-Void Grid \PoM{}}\label{nonuniqueconv.sub.}
Let $\psi=\Z\backslash 4\Z$; i.e., those integers which are not multiple
of $4$. If $U$ is a uniform random variable in $[0,4)$, then $\psi+U$
is a stationary point process on the real line which will be called the
quadri-void grid below. The Palm distribution of this point process
has mass of $\frac 13$ on $\theta_{1}\psi,\theta_{2}\psi$ and $\theta_{3}\psi$.

Let $q$ be the \pomk{} defined by 
\eqn{q({\theta_{1}\psi})=2 \text{ , } 
q({\theta_{2}\psi})=1\text{ and }q({\theta_{3}\psi})= -2.}
For odd values of $n>0$, one has 
\eqn{{\cal P}_0^{q,n}[\phi=\theta_{3}\psi]=\frac 23,\quad {\cal P}_0^{q,n}[\phi=\theta_{1}\psi]=\frac 13,}
whereas for even values of $n>0$,
\eqn{{\cal P}_0^{q,n}[\phi=\theta_{3}\psi]=\frac 13,\quad {\cal P}_0^{q,n}[\phi=\theta_{1}\psi]=\frac 23.}
Therefore $({\cal P}_0^{q,n})_{n=1}^\infty$ 
has two convergent subsequences with different limits,
one for even and one for odd values of $n$,
and none of these limits is invariant under the action of 
$(\theta_q)_*$.
However, the sequence $(\widetilde {\cal P}^{q,n}_0)_{n=1}^\infty$
converges to a limit $\widetilde {\cal P}^q_0$ which is the mean
of the odd and even $g$-probabilities, i.e., 
\eqn{\widetilde {\cal P}^q_0[\phi=\theta_{3}\psi]=\frac 12,\quad \widetilde {\cal P}^q_0[\phi=\theta_{1}\psi]=\frac 12,}
and it is invariant under the action of $(\theta_q)_*$.

\appendix

\section{Random Measures}\label{RandMeas}
This subsection summarizes the results about random measures which are used in this paper in order to have a self-contained paper. The interested reader should  refer to \cite{Ka76,Ka02}. No proofs are given.

Let $S$ be a locally compact (all points have a compact neighborhood) second countable (has a countable base) Hausdorff space. In this case, $S$ is known to be Polish, i.e., there exists some separable and complete metrization $\rho$ of $S$.

Let $\bo(S)$ be the Borel algebra of $S$ and $\bob(S)$ be all bounded elements of $\bo(S)$; i.e., all $B\in\bo(S)$ such that the closure of $B$ is compact. Let $\m(S)$ be the class of all Radon measures on $(S,\bo(S))$; i.e., all measures $\mu$ such that for all $B\in\bob(S)$, $\mu B<\infty$ and let $\n(S)$ be the subspace of all $\Bbb N$-valued measures in $\m(S)$.
 The elements of $\n(S)$ are {\em counting measures}.
For all $\mu$ in $\m(S)$, define 
\eqn{\bob(S)^\mu:=\{B\in\bob(S);\mu(\partial B)=0\}.}
Let  $C_b(S)$ (respectively $C_c(S)$) be the class of all continuous and bounded (respectively continuous and compact support) $h:S\to\R^{+}$. Let   
\eqn{\mu h:=\int_S h(x)\mu(\d x),}
where the latter is equal to $\sum_{x\in\mu}h(x)$ when $\mu$ is a counting measure. Note that in the summation one takes the  multiplicity of points into account. 
The class of all finite intersections of $\m(S)$-sets (or $\n(S)$-sets) of the form $\{\mu:s<\mu h<t\}$ with real $r$ and $s$ and arbitrary $h\in C_c(S)$ forms a base of a topology on $\n(S)$ which is known as the  \emph{vague topology}. In the vague topology $\n(S)$ is closed in $\m(S)$ (\cite{Ka76}, p. 94, A 7.4.).
A necessary and sufficient condition for the convergence in this topology (\cite{Ka76}, p. 93) is:
\eqn{\mu_n\tov\mu \Leftrightarrow \forall h\in C_c(S),\  \mu_n h\to\mu h.}
If one considers the subspace of all bounded measures in $\n(S)$, one may replace $C_c(S)$ by $C_b(s)$. This leads to the \emph{weak topology} for which 
\eqn{\mu_n\tow \mu \Leftrightarrow \forall h\in C_b(S),  \mu_n h\to\mu h.}
The convergence in distribution of the random variables
$\xi_1,\xi_2,\ldots$, defined on $(\Omega, \cal F, \p)$
and taking their values in $(S,\bo(S))$,
to the random element $\xi$ is defined as follow
\eqn{\xi_n\tod \xi\Leftrightarrow(\xi_n)_*\p\tow (\xi)_*\p.}

The next lemma describes the relation between the convergences in the vague topology and the weak one. 

\begin{lem}[\cite{Ka76}, p.95, A 7.6.]\label{weak.vague}
For all bounded $\mu,\mu_1,\mu_2,\ldots\in\m(S)$, one has
\eqn{\mu_n\tow\mu\Leftrightarrow\mu_n\tov\mu \text{ and } \mu_nS\to\mu S.} 
\end{lem}

According to Lemma~\ref{weak.vague}, when discussing the convergence of probability measures, there is no difference between the vague and the weak convergence.  

The following proposition is a key point in the development of the theory of random measures and random point processes (\cite{Ka76}, p. 95 A 7.7.). 

\begin{prop}\label{vague.p}
Both $\m(S)$ and $\n(S)$ are Polish in the vague topology. Also the subspaces of bounded measures in $\m(S)$ and $\n(S)$ are Polish in the weak topology.
\end{prop}

Proposition~\ref{vague.p}  allows one to define measures on $\m(S)$ or $\n(S)$ which are Polish spaces and use for them the theory available for $S$. If $\cal M$ (respectively $\N$) is the $\sigma$-algebra generated by the vague topology on $\m(S)$ (respectively $\n(S)$), a {\em random measure} (respectively {\em random point process}) on $S$ is simply a random element of $(\m(S),{\cal M})$ (respectively$(\n(S),\N)$). Note that a random point process is a special case of a random measure. 

The next theorem  and  lemmas give handy tools to deal with convergence in distribution of random measures on $S$.   
\begin{thm}[\cite{Ka76}, p.22, Theorem 4.2.]
If $\mu, \mu_1,\mu_2,\ldots$ are random measures on $S$ (i.e., random elements of $(\m(S),{\cal M)}$), then 
\eqn{\mu_n\tod\mu\Leftrightarrow\mu_n h\tod\mu h, \quad \forall h\in C_c(S).}
\end{thm}
\begin{lem}[\cite{Ka76}, p.22, Lemma 4.4.]\label{lem:discont.meas.zero}
If $\mu,\mu_1,\mu_2,\ldots$ are random measures on $S$ satisfying $\mu_n\tod\mu$, then $\mu_nh\tod\mu h$ for every bounded measurable function $h:S\to\R^+$ with bounded support satisfying $\mu (D_h)=0$ almost surely, where $D_h$ is the set of all discontinuity points of $h$.   Furthermore, 
\eqn{(\mu_n B_1,\ldots, \mu_n B_k)\tod (\mu B_1, \ldots \mu B_k),\qquad k\in \Bbb N, \quad B_1,\ldots B_k\in\bob(S)^\mu.}
\end{lem}
\begin{lem}[\cite{Ka76}, p.23, Lemma 4.5.]
A sequence $(\mu_n)_{n=1}^\infty$ of random measures on $S$ is relatively compact w.r.t. the convergence in distribution in the vague topology if and only if 
\eqn{\lim_{t\to\infty}\limsup_{n\to\infty}\p[\mu_nB>t]=0,\quad \forall B\in\bob(S).}
\end{lem}
Denote by $P(S)$ the set of all probability measures on $S$. Clearly $P(S)\subset M(S)$ and according to Lemma~\ref{weak.vague}, the weak and the vague topologies on $P(S)$ coincide. 
\begin{prop}[\cite{Bi68} p.30, Theorem 5.1.]\label{cont.m.cont}
If $S$ and $T$ are Polish spaces and $h:(S,\bo(S))\to (T,\bo(T))$ is a measurable mapping, then $h_*$ 
is continuous w.r.t. the weak topology
at point $\p\in P(S)$
if $h$ is $\p$-almost surely continuous.
\end{prop}
Note that the version of Proposition~\ref{cont.m.cont} which is in \cite{Bi68}, is expressed for metric spaces. But, as noted in the beginning of the appendix, Polish spaces are metrizable and hence one can apply the same statement for such spaces. 

\section{Semigroup Actions} 
Let $X$ be a Hausdorff space.
An action of $(\bbN,+)$ on $X$ is a collection $\pi$
of mappings $\pi_n:X\to X$, $n\in \bbN$, such that
for all $x\in X$, and $m,n\in \bbN$,
$\pi_m  \circ \pi_n(x)= \pi_{m+n} (x)$.
When each of the mappings $\pi_n$ is continuous,
$\pi$ is also often referred to as a discrete time dynamical system.

On a Hausdorff space $X$, one can
endow the set $X^X$ with a topology,
e.g. that of pointwise convergence.
The closure of the action of $\bbN$
is then the closure $\overline \Pi$ of the set
$\Pi=\{\pi_n,n\in \bbN\}\subset X^X$
w.r.t. this topology. A classical instance (see e.g. 
\cite{ElElNe00}) is that where the space $X$ is compact, the
mappings $\pi_n$ are all continuous,
and the topology on $X^X$ is that of pointwise convergence.
Then $\overline \Pi$ is compact.

Denote the orbit $\{x,\pi(x),\pi_2(x),\cdots\}$
of $x\in X$ by $A_x$.
For all $x\in X$, the
closure $ \mathrm{cl}{A_x}$ of
$A_x$ is a closed $\pi$-invariant
subset of $X$. If, for all $n$, $\pi_n$ is continuous,
then the restriction of $\pi$ to $\mathrm{cl} {A_x}$
defines a semigroup action of $\bbN$.
The compactness of $\mathrm{cl} {A_x}$ is not granted when 
$X$ is non-compact. When it holds,
several important structural
properties follow as illustrated by the next lemmas
where $X$ is a metric space with distance $d$. Let
\begin{equation}\label{eq:omegalim}
\omega_x=\{y\in X \mbox{ s.t. } \exists n_1<n_2<\cdots \in \bbN 
\mbox{ with } \pi_{n_i} (x) \to y\}
\end{equation}
denote the {\em $\omega$-limit set} of $x$.
\begin{lem} [Lemma 4.2, p. 134, and p. 166 in \cite{BrTa10}]
\label{lemcherche}
Assume that $\pi_n$ is continuous
for all $n$ and that $\mathrm{cl} {A_x}$ is compact.
Then, for all neighborhoods $U$ of $\omega_x$, there exists
an $N=N(U,x)$ such that $\pi_n(x)\in U$ for all $n\ge N$.
Moreover $\omega_x$ is non empty,
compact and $\pi$-invariant.
\end{lem}
In words, under the compactness and continuity conditions,
the orbit is attracted to the $\omega$-limit set.
\begin{lem} [Lemma 2.9, p. 95 in \cite{BrTa10}]
\label{lemtak}
If $\mathrm{cl} {A_x}$ is compact, then the 
following property holds:
for all $\epsilon >0$, there exists $N= N(\epsilon,x)\in \bbN$ such that
for all $y\in \mathrm{cl} {A_x}$, the set $\{\pi_n(x), 0\le n\le N\}$ contains
a point $z$ such that $d(y,z)\le \epsilon$.
If in addition $\pi_n$ is continuous 
for all $n$, then the last property
is equivalent to the compactness of $\mathrm{cl} {A_x}$.
\end{lem}
In words, under the compactness condition,
in a long enough interval, the trajectory $\pi_n(x)$ visits a neighborhood
of every point of $\mathrm{cl} {A_x}$.

\section{Proof of Proposition \ref{contpos}}
\label{seccontpos}
Let $g$ be a \pomk{} the image of which at $\phi\in\n^0$ is
$x\in\phi$, with $x\neq 0$. Assume there is a point
$y\in\phi$ with $y\notin\{0,x\}$. Since $\phi$ is a discrete
subset of $\sp$ and $d\geq 2$ there exist curves 
$\gamma_1,\gamma_2:[0,1]\to\sp$ such that 
\begin{enumerate}
\item \label{ends} $\gamma_1(0)=\gamma_2(1)=x$ and $\gamma_2(0)=\gamma_1(1)=y$;
\item \label{intersections}$\gamma_1$ and $\gamma_2$ only intersect at their end-points; 
\item \label{points}$\gamma_1$ and $\gamma_2$ contain no point of $\phi$ other than $x$ and $y$. 
\end{enumerate}
Now let $\Gamma$ be a closed curve in $\n^0$ defined as 
\eqn{\Gamma:[0,1]\to\n^0;\quad
\Gamma(t)=(\phi\backslash \{x,y\})\cup \{\gamma_1(t),\gamma_2(t)\},
\ t\in[0,1].}
The continuity of $g$, \ref{intersections}. and \ref{points}. imply
that for all $t\in[0,1]$, $g(\Gamma(t))=\gamma_1(t)$. 
Hence $g(\Gamma(0))=x$ and $g(\Gamma(1))=y$.
But it follows from \ref{ends}. that $\Gamma(0)=\Gamma(1)=\phi$,
which contradicts the fact that $x$ and $y$ are different
points of $\phi$. When $\phi=\{0,x\}$, one obtains
the contradiction by letting $x$ go to infinity whereas in this situation,
$\{0,x\}$ converges to $\{0\}$ in the vague topology.  

\section*{Acknowledgments}
The authors would like to thank H. Thorisson, K. Alishahi, A. Khezeli
and A. Sodre, as well as the anonymous reviewer, for their
very valuable comments on this work.
The early stages of this work were initiated 
at Ecole Normale Sup\'erieure and INRIA,
where they were supported by a grant
from Minist\`ere des Affaires Etrang\`eres. 
The later stages were pursued 
at the University of Texas at Austin and were supported
by a grant of the Simons Foundation 
(\#197982 to UT Austin). The second author expresses
his gratitude to the higher administration of Sharif University of Technology,
especially to S.-G. Miremadi, for their crucial support. 

\bibliographystyle{amsplain}
\bibliography{bib-01-16}

\end{document}